\documentclass{amsart}

 \usepackage{amsmath,amsfonts,amsthm,amssymb,graphics,graphicx, verbatim}

 \newtheorem{thm}{Theorem}[section]

 \newtheorem{prop}[thm]{Proposition}
 \newtheorem{example}[thm]{Example}

\newcommand{\norm}[1]{\left|\!\left|{#1}\right|\!\right|}
 \newcommand{\R}{\ensuremath{\mathbb{R}}}

 \newcommand\tlide[1]{\tilde{#1}}

\newcommand\Id{\operatorname{Id}}

\title[Sharp families of  eigenfunctions and quasimodes]{A note on constructing families of sharp examples for $L^{p}$ growth of eigenfunctions and quasimodes}

\author{Melissa Tacy}

\email{mtacy@maths.otago.ac.nz}

\address{School of Mathematics and Statistics, University of Otago, Otago, 9016 New Zealand}

\begin{document}

  \begin{abstract}
In this note we analyse $L^{p}$ estimates for Laplacian eigenfunctions and quasimodes and their associated sharp examples. In particular we use previously determined estimates to produce a new set of estimates for restriction to thickened neighbourhoods of submanifolds. In addition we produce a family flat model quasimode examples that can be used to determine sharpness of estimates on Laplacian eigenfunctions restricted to subsets. For each quasimode in the family we show that there is a corresponding spherical harmonic that displays the same growth properties. Therefore it is enough to check $L^{p}$ growth  estimates against the  simple flat model examples. Finally we present a heuristic that for any subset determines which quasimode in the family is expected to produce sharp examples. 
 \end{abstract}
\maketitle

Let $(M,g)$ be a Riemannian manifold and $\Delta=\Delta_{g}$ be the (positive) Laplace-Beltrami operator defined by the metric. There has been much recent interest (for example \cite{BGT},\cite{Chen14} ,\cite{HTacy}, \cite{koch}, \cite{S88},\cite{soggesurvey},\cite{tacy09}) in understanding how the $L^{p}$ norms of Laplacian eigenfunctions
$$\Delta{}u=\lambda^{2}u$$
grow for large $\lambda$. In particular in comparing  the $L^{p}$ estimates over the full manifold with that on subsets. The results in this area produce estimates of the form
$$\norm{u}_{L^{p}(X)}\lesssim{}\lambda^{\delta(n,p,X)}\norm{u}_{L^{2}(M)}$$
where $X$ is a subset of $M$ (not necessarily of full dimension). Is is often instructive to translate to a semiclassical problem where $\lambda^{-1}=h$ and $u$ is a solution to the semiclassical equation $(h^{2}\Delta-1)u$. In fact, for a number of technical reasons, it is more usual to consider approximate solutions, that is $u$ such that
$$\norm{(h^{2}\Delta-1)u}_{L^{2}(M)}\lesssim{}h\norm{u}_{L^{2}(M)}.$$

The purpose of this note is twofold
\begin{enumerate}
\item To examine the known estimates and associated sharp examples and obtain new sharp estimates by ``cheap'' techniques (such as the application of H\"{o}lder's inequality or interpolation).
\item    To describe how to construct families of examples to examine questions of sharpness both for the flat model cases and for spherical harmonics. 
\end{enumerate}
In particular we will obtain $L^{p}$ estimates where $X$ is a thickened region of a submanifold. The examples we construct will show that these estimates are sharp (up to a possible $\log$ loss). The spherical harmonic examples have the advantage of being exact eigenfunctions however they are not so easy to write down explicitly. The flat model examples are in contrast very easy to explicitly produce. The flat model also has the advantage that for any given $p$ (with knowledge of the semiclassical version of the $L^{p}$ estimate proof) it is easy to determine which functions in the family will give rise to sharp examples.  We will show that every flat model example has a matching spherical harmonic which shares all relevant features. Therefore any result that is sharp under the flat model is sharp under spherical harmonics. Finally we discuss how, given any particular $p$, one predicts which example will give rise to sharp estimates. 

The whole and submanifold estimates are as follows
$$\norm{u}_{L^{p}(X)}\lesssim{}h^{-\delta(n,k,p)}\norm{u}_{L^{2}(M)}$$
for $X$ a $k$ dimensional smooth submanifold if $k<n$ and for $X=M$ if $k=n$. The function $\delta(n,k,p)$ is given by
$$\delta(n,n,p)=\begin{cases}
\frac{n-1}{2}-\frac{n}{p}&\frac{2(n+1)}{n-1}\leq{}p\leq{}\infty\\
\frac{n-1}{4}-\frac{n-1}{2p}&2\leq{}p\leq{}\frac{2(n+1)}{n-1}\end{cases}$$
$$\delta(n,n-1,p)=\begin{cases}
\frac{n-1}{2}-\frac{n-1}{p}&\frac{2n}{n-1}\leq{}p\leq{}\infty\\
\frac{n-1}{4}-\frac{n-2}{2p}&2\leq{}p\leq{}\frac{2n}{n-1}\end{cases}$$
and for $k\leq{}n-2$
$$\delta(n,k,p)=\frac{n-1}{2}+\frac{k}{p}\quad{}2<p<\infty.$$
In the case $k\leq{}n-3$ or $n=3,k=2$ the $p=2$ estimate is included, otherwise there is a logarithmic loss
$$\norm{u}_{L^{2}(X)}\lesssim{}h^{-\frac{n-1+k}{2}}\log|h|\norm{u}_{L^{2}}.$$
These estimates are due to
\begin{itemize}
\item Sogge \cite{S88} for $L^{p}$ estimates over the full manifold and Koch-Tataru-Zworski \cite{koch} for the semiclassical problem.
\item Burq-G\'{e}rard-Tzevtkov \cite{BGT} for $L^{p}$ estimates of eigenfunctions on submanifolds and Tacy \cite{tacy09} for the semiclassical problem.
\item Chen-Sogge \cite{Chen14} for the endpoint estimate $(n,k,p)=(3,2,2)$.
\end{itemize}
It is well known that these estimates are saturated for high $p$ by the zonal harmonics and for low $p$ by the highest weight harmonics. The key features that saturate the estimates are a point concentration and a tube concentration. Zonal harmonics have a point concentration at their north pole while highest weight harmonics are highly concentrated in a $h^{\frac{1}{2}}$ width tube around a great circle.

\begin{figure}[h!]\label{pointfig}
\includegraphics[scale=0.4]{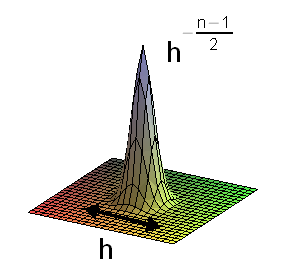}
\caption{Concentration at a point}
\end{figure}

\begin{figure}[h!]\label{tubefig}
\includegraphics[scale=0.4]{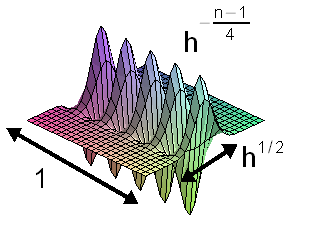}
\caption{Concentration in a tube}
\end{figure}

These two examples alone are sometimes enough to analyse sharp $L^{p}$ behaviour. We demonstrate this for restriction of eigenfunctions to sets near submanifolds. 
For $\Sigma$ a smooth $k$ dimensional submanifold of $M$ let $\Sigma_{\beta}$ be the set
$$\Sigma_{\beta}=\{x\in{}M\mid{}d(x,\Sigma)\leq{}h^{\beta}\}$$
where $d$ is the usual distance associated with the metric $g$. We want an estimate of the form
$$\norm{u}_{L^{p}(\Sigma_{\beta})}\lesssim{}h^{-\sigma(n,k,p,\beta)}\norm{u}_{L^{2}(M)}$$
for $u$ an Laplacian eigenfunction or of quasimode of $h^{2}\Delta_{g}-1$. We first  make some observations using prior results and the point/tube examples. These observations will enable us to determine $\sigma(n,k,p,\beta)$ in many cases.
\begin{itemize}
\item[Observation 1] Clearly the $L^{p}$ norm of $u$ on $\Sigma_{\beta}$ must be bounded by the $L^{p}$ norm of $u$ on $M$. Therefore for all $p$ we have
\begin{equation}\norm{u}_{L^{p}(\Sigma_{\beta})}\lesssim{}h^{-\delta(n,n,p)}\norm{u}_{L^{2}(M)}.\label{trivfull}\end{equation}
The question is then whether this can be improved. We must therefore first ask whether the known point and tube features fit inside $\Sigma_{\beta}$. If they do we can expect no better estimates that \eqref{trivfull}.
\item[Observation 2] Since both the tube and point features can be placed inside $\Sigma_{\beta}$ where $\beta\leq{}\frac{1}{2}$ we know immediately that there can be no better estimates in this case.
\item[Observation 3] Writing $x\in{}M$ as $x=(y,z)$ where $\Sigma=\{(y,z)\in{}M\mid{}z=0\}$ we see that
$$\int_{\Sigma_{\beta}}|u|^{p}dx\leq{}\sup_{|z|\leq{}ch^{\beta}}\int{}|u|^{p}dy\times\int_{|z|\leq{}ch^{\beta}}dz$$
$$\lesssim{}h^{-p\delta(n,k,p)}h^{\beta(n-k)}.$$
So we may also say that
\begin{equation}\norm{u}_{L^{p}(\Sigma_{\beta})}\lesssim{}h^{-\delta(n,k,p)+\frac{\beta(n-k)}{p}}\norm{u}_{L^{2}(M)}.\label{trivsub}\end{equation}
If for all $|z|\leq{}h^{\beta}$ the submanifold estimate is sharp we cannot expect to do better than \eqref{trivsub}.
\item[Observation 4] Eigenfunctions and quasimodes have the property that they oscillate with frequency on the order $h^{-1}$, therefore the cannot change much in a region of size $h$. This means that the estimate of \eqref{trivsub} is the best we may expect for $\beta\geq{}1$ and in fact we see that this is indeed the cases for both the point and tube sharp examples.
\end{itemize}  

From these four observations (along with interpolation from known results) we, in Section 1, generate a full set of $L^{p}(\Sigma_{\beta})$  In Section \ref{flatexample} we construct a family of sharp quasimodes examples in the flat model case that prove these $L^{p}$ estimates to be sharp.  In Section \ref{quastoeig} we show that on the sphere we can construct exact eigenfunctions with the same properties as the sharp quasimode examples which means that, in any situation, we may check estimates against the flat model. In Section \ref{pickexample} we discuss how, given knowledge of the semiclassical techniques employed to prove the whole and submanifold estimates, one chooses the correct example to get a sharp quasimode.

 \subsection*{Acknowledgements} The author acknowledges the comments and suggestions of the reviewer that greatly improved the exposition of this paper.

\section{$L^{p}$ estimates on $\Sigma_{\beta}$}\label{Lpest}

In this section we use our four observations along with known results to prove a full range of $L^{p}$ estimates for $\Sigma_{\beta}$. Taken together Observations 2 and 4 tell us that there are no non-trivial estimates outside $\frac{1}{2}\leq{}\beta\leq{}1$ so we focus on this region. In the case that $\Sigma$ is a hypersurface we obtain the following bounds
\begin{thm}\label{hypersurface:thm} Suppose $u$ is an $O_{L^{2}}(h)$ quasimode of $h^{2}\Delta-1$ on a Riemannian manifold $(M,g)$.  Further for $\Sigma$ a smooth embedded hypersurface in $M$ and $\frac{1}{2}\leq{}\beta\leq{}1$, let $\Sigma_{\beta}=\{x\in{}M\mid{}d(x,\Sigma)\leq{}h^{\beta}\}$. Then
$$\norm{u}_{L^{p}(\Sigma_{\beta})}\lesssim{}h^{-\sigma(n,n-1,p)}\norm{u}_{L^{2}(M)}$$
where
$$\sigma(n,n-1,p,\beta)=\begin{cases}
\delta(n,n,p)&p\geq{}\frac{2(n+1)}{n-1}\\
\frac{\beta(n-1)}{2}-\frac{\beta(n+1)}{p}+\frac{1}{p}&\frac{2n}{n-1}\leq{}p\leq{}\frac{2(n+1)}{n-1}\\
\delta(n,n-1,p)-\frac{\beta}{p}&2\leq{}p\leq{}\frac{2n}{n-1}.\end{cases}$$
\end{thm}

\begin{proof}

From Observation 1 we know that $u$ must obey the full manifold estimates. Since $\beta\geq 1$ we may always fit the point type example into $\Sigma_{\beta}$ so we cannot expect better estimates than those arising from a point concentration. So we know that if $p\geq{}\frac{2(n+1)}{n-1}$ we cannot expect better estimates than those from over the full manifold. That is
$$\sigma(n,n-1,p,\beta)=\delta(n,n,p)\quad{}p\geq{}\frac{2(n+1)}{n-1}.$$
The sharp example for the low $p$ (that is $2\leq{}p\leq{}\frac{2n}{n-1}$) hypersurface estimates is the tube oriented with its long direction along the hypersurface. Since this examples has relatively constant size in a $1\times{}h^{\frac{n-1}{2}}$ region the estimates of Observation 3 are the best we could expect in this range of $p$. That is
$$\sigma(n,n-1,p,\beta)=\delta(n,n-1,p)-\frac{\beta}{p}\quad{}2\leq{}p\leq{}\frac{2n}{n-1}.$$
 Therefore the only unknown estimates are those between $\frac{2n}{n-1}$ and $\frac{2(n+1)}{n-1}$. We interpolate between the estimate for $p\frac{2(n+1)}{n-1}$ and $p=\frac{2n}{n-1}$ to obtain $$\sigma(n,n-1,p,\beta)=\begin{cases}
\delta(n,n,p)&p\geq{}\frac{2(n+1)}{n-1}\\
\frac{\beta(n-1)}{2}-\frac{\beta(n+1)}{p}+\frac{1}{p}&\frac{2n}{n-1}\leq{}p\leq{}\frac{2(n+1)}{n-1}\\
\delta(n,n-1,p)-\frac{\beta}{p}&2\leq{}p\leq{}\frac{2n}{n-1}\end{cases}$$
Since Observations 1 and 3 (along with the known sharp examples for manifolds and hypersurfaces) tell us that we have sharp examples for $p\geq{}\frac{2(n+1)}{n-1}$ and $p\leq{}\frac{2n}{n-1}$ the only question remaining is whether the intermediate bounds obtained through interpolation are sharp. In Section \ref{flatexample} we will construct model quasimodes that demonstrate sharpness. The results of Section \ref{quastoeig} guarantee that there are exact eigenfunctions on the sphere that are also sharp.

\end{proof}

Where $\Sigma$ is a lower dimensional submanifold we obtain.

\begin{thm}\label{lower:thm} Suppose $u$ is an $O_{L^{2}}(h)$ quasimode of $h^{2}\Delta-1$ on a Riemannian manifold $(M,g)$.  Further for $\Sigma$ a smooth embedded submanifold of dimension $k\leq{}n-3$ in $M$ and $\frac{1}{2}\leq{}\beta\leq{}1$, let $\Sigma_{\beta}=\{x\in{}M\mid{}d(x,\Sigma)\leq{}h^{\beta}\}$. Then
$$\norm{u}_{L^{p}(\Sigma_{\beta})}\lesssim{}h^{-\sigma(n,k,p)}\norm{u}_{L^{2}(M)}$$
where
$$\sigma(n,k,p,\beta)=\begin{cases}
\delta(n,n,p)&p\geq{}\frac{2(n+1)}{n-1}\\
\frac{\beta(n-1)}{2}-\frac{\beta(n-1)}{p}+\frac{1}{p}&2\leq{}p\leq{}\frac{2(n+1)}{n-1}.\end{cases}$$
If $k=n-2$ the same result holds for $p\geq{}\frac{2(n+1)}{n-1}$ and holds with a $\log$ loss when $p<\frac{2(n+1)}{n-1}$.
\end{thm}

\begin{proof}
Again Observation 1 along with the point sharp example tells us that if $p\geq{}\frac{2(n+1)}{n-1}$ we cannot expect better estimates than those from the full manifold. Therefore 
$$\sigma(n,k,p,\beta)=\delta(n,n,p)\quad{}p\geq{}\frac{2(n+1)}{n-1}$$
The sharp submanifold restriction examples however are the point type eigenfunctions. This feature only persists for an $O(h)$ region so when $\beta\ll 1$ we cannot expect to get sharp examples for low $p$ from Observation 3. However Burq and Zuily \cite{Burq2016} have, in the case $k\leq{}n-3$, obtained
$$\norm{u}_{L^{2}(\Sigma_{\beta})}\lesssim{}h^{\beta-\frac{1}{2}}\norm{u}_{L^{2}(M)}$$
and that when $k=n-2$ the same result holds with a $\log$ loss.
so we may interpolate from this point to obtain
$$\sigma(n,k,p,\beta)=\begin{cases}
\delta(n,n,p)&p\geq{}\frac{2(n+1)}{n-1}\\
\frac{\beta(n-1)}{2}-\frac{\beta(n-1)}{p}+\frac{1}{p}&2\leq{}p\leq{}\frac{2(n+1)}{n-1}\end{cases}$$
 We know that the high $p$ estimates are sharp. In Sections \ref{flatexample} and \ref{quastoeig} we show that the low  $p$ estimates (modulo the log loss) are also sharp.
\end{proof}

\section{Flat model examples}\label{flatexample}

 We study the flat model, that is, localised quasimodes of the Laplacian in $\R^n$, to gain insight into sharp examples. Such quasimodes can be produced on the Fourier side easily. In keeping with the semiclassical theme we use the re-scaled semiclassical Fourier transform,
$$\mathcal{F}_{h}[u](\xi)=\frac{1}{(2\pi{}h)^{n/2}}\int_{\R^n}e^{-\frac{i}{h}\langle{}x,\xi\rangle}u(x)\,dx.$$ 
This operator has the property that
$$\mathcal{F}_{h}\left[hD_{x_{i}}\right]={}\xi_{i}\mathcal{F}_{h}[u]$$
and 
$$\norm{\mathcal{F}_{h}[u]}_{L^{2}}=\norm{u}_{L^{2}}.$$
The development of flat model examples was discussed in \cite{Guo15}. We include it here for the readers convenience.

Suppose that $u$ is an $L^{2}$ normalised $O_{L^2}(h)$ quasimode of $\Delta_{\R^n}$. We must have
$$\norm{(|\xi|^{2}-1)\mathcal{F}_{h}[u]}_{L^{2}(\R^{n})}\lesssim{}h,$$
Thus $\mathcal{F}_{h}[u]$ must be located near the sphere of radius $1$ in the $\xi$-variables. We create a family of quaismodes indexed by $\alpha$ which controls the degree of angular dispersion of $\xi$. Write $\xi=(r,\omega)$ where $\omega\in{}S^{n-1}$ and set the coordinate system so that $\omega_{0}$ corresponds with the unit vector in the $\xi_{1}$ direction. Let
$$\chi_{\alpha}^{h}(r,\omega)=\begin{cases}
1 & \text{if }|r-1|<h,|\omega-\omega_{0}|<h^{\alpha},\\
0 & \mbox{otherwise}.\end{cases}$$
Then set
$$f^h_{\alpha}(\xi)=f^{h}_{\alpha}(r,\omega)=h^{-1/2-\alpha(n-1)/2}\chi(r,\omega).$$
Note that $f^{h}_{\alpha}$ is $L^{2}$ normalised. Now set
$$T^h_{\alpha}(x)=\mathcal{F}_{h}^{-1}[f^h_{\alpha}](x)=\frac{1}{(2\pi{}h)^{n/2}}\int_{\R^n}e^{\frac{i}{h}\langle{}x,\xi\rangle}f_{\alpha}(\xi)\,d\xi.$$
$T^h_{\alpha}$ is an $L^{2}$ normalised $O(h)$ quasimode of $\Delta_{\R^n}$. We may write
$$T^h_{\alpha}(x)=\frac{h^{-1/2-\alpha(n-1)/2-n/2}e^{\frac{i}{h}x_{1}}}{(2\pi)^{n/2}}\int_{\R^n}e^{\frac{i}{h}(x_{1}(\xi_{1}-1)+\langle{}x',\xi'\rangle)}\chi_{\alpha}(\xi)\,d\xi.$$
Note that if $|x_{1}|<\epsilon{}h^{1-2\alpha}$ and $|x'|<\epsilon{}h^{1-\alpha}$ for sufficiently small $\epsilon>0$, the factor
$$e^{\frac{i}{h}(x_{1}(\xi_{1}-1)+\langle{}x',\xi'\rangle)}$$
does not oscillate so in this region
$$|T^h_{\alpha}(x)|>ch^{-(n-1)/2+\alpha(n-1)/2}.$$
\begin{figure}[h!]\label{Talphafig}
\includegraphics[scale=0.3]{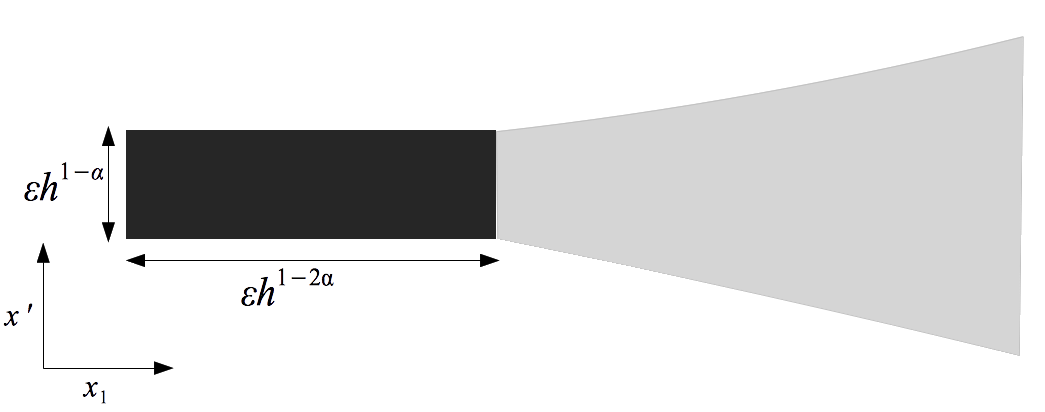}
\caption{$T^{h}_{\alpha}$ is localised so that is large in a $h^{1-2\alpha}\times{}(h^{1-\alpha})^{n-1}$ tube}
\end{figure}
We claim that when $\alpha=1-\beta$ the function $T_{\alpha}$ saturates the $L^{p}$ estimates for $\Sigma_{\beta}$ in the case where $\Sigma$ is a hypersurface and $\frac{2n}{n-1}\leq{}p\leq\frac{2(n+1)}{n-1}$ as well as the case where $\Sigma$ is a lower dimensional submanifold and $p\leq{}\frac{2(n+1)}{n-1}$.

\begin{example}\label{example}
We choose coordinates such that when we write $x\in{}M$ as $x=(y,z)$, $\Sigma=\{(y,z)\in{}M\mid{}z=0\}$. By setting $\alpha=1-\beta$ we produce a function that has a $h^{2\beta-1}\times{}h^{\beta}$ tube where 
$$|T_{1-\beta}|\geq{}ch^{-\frac{\beta(n-1)}{2}}$$
Rotations and translations of $T_{\alpha}$ are still quasimodes so we may align it so that the long direction lies in the submanifold. Therefore
we have
$$\norm{T_{\alpha}}_{L^{p}(\Sigma_{\beta})}>ch^{-\frac{\beta(n-1)}{2}}h^{\frac{\beta(n-1)}{p}+\frac{2\beta-1}{p}}$$
$$=ch^{-\frac{\beta(n-1)}{2}+\frac{\beta(n+1)}{p}-\frac{1}{p}}$$
as required.

\end{example} 

One could obtain this example by calculating lower bounds for the $L^{p}$ norm for every $\alpha$ and then maximising. However by understanding the heuristics of the semiclassical proof one can immediately select the correct scale to find sharp examples in any situation. We discuss this heuristic in Section \ref{pickexample}.

\section{From quasimodes to exact eigenfunction}\label{quastoeig}

While they are easy to work with the quasimodes $T_{\alpha}$ only show us that estimates are sharp for quasimodes of the flat Laplacian. However we can construct exact $L^{2}$ normalised eigenfunctions $\phi_{\alpha}$, on the sphere  that have all the relevant  properties of $T_{\alpha}$. That is they have a $h^{1-2\alpha}\times{}h^{(1-\alpha)(n-1)}$ region where $|\phi_{\alpha}|\geq{}ch^{-\frac{n-1}{2}+\frac{\alpha(n-1)}{2}}$. Since this is the only property of $T_{\alpha}$ used to prove sharp examples this construction shows that any sharp examples from $T_{\alpha}$ give rise to sharp examples of exact eigenfunctions on the sphere. So any quasimode estimates that are sharp for the family of flat quasimode examples $T_{\alpha}$ are also sharp (with exact eigenfunctions) on the sphere. 

To understand which spherical harmonics to pick we first re-express $T_{\alpha}$ as a sum of quasimodes, each of which has a  Fourier transform localised in the angular variables on the scale of $h^{1/2}$.  This is the localisation scale of $T_{1/2}$. Note that $T_{1/2}$ is localised about the point $(1,0,\dots,0)$. We can produce a function $T^{j}_{1/2}$ with Fourier support in a $h\times{}h^{\frac{n-1}{2}}$ region of any $\xi_{j}\in S^{n-1}$ by a rotation applied to $\mathcal{F}_{h}(T_{1/2})$. The quasimode produced by this rotation  is simply the standard $T_{1/2}$ quasimode rotated so that the long axis lies along $\xi_{j}$. Now $f_{\alpha}^{h}(\xi)$ is supported in an $h^{\alpha(n-1)}$ angular region so we can cover this support with $h^{\left(\alpha-\frac{1}{2}\right)(n-1)}$ rotations of $\mathcal{F}_{h}[T_{1/2}]$. Therefore $T_{\alpha}$ can be though of as a sum of $h^{\left(\alpha-\frac{1}{2}\right)(n-1)}$ functions each of which is a rotation of $T_{1/2}$.

The flat Laplacian quasimodes $T_{1/2}$ resemble the tubular concentrations we see in highest weight spherical harmonics. This leads us to the idea that we can create suitable a $\phi_{\alpha}$ by considering a sum of rotated highest weight spherical harmonics. We write $S^{n}$ as the subset of $\R^{n+1}$ where $|x|=1$, it is well known that the function
$$u(x)=j^{\frac{n-1}{4}}(x_{1}+ix_{2})^{j}$$
is a solution to the spherical Laplacian eigenfunction equation with $j(j+n-1)=\lambda^{2}=h^{-2}$. Further if $x=(x_{1},x_{2},\bar{x})$ then
$$|u(x)|^{2}=j^{\frac{n-1}{2}}(1-|\bar{x}|^{2})^{j}=j^{\frac{n-1}{2}}e^{j\log(1-|\bar{x}|^{2})}$$
so $u(x)$ is highly concentrated on the equation $\bar{x}=0$ with exponential decay  when $|\bar{x}|\gg h^{1/2}$. The pre-factor of $j^{\frac{n-1}{4}}\approx h^{-\frac{n-1}{4}}$ ensures that $\norm{u}_{L^{2}}\approx{}1$. We produce an example by summing rotations of $u(x)$.

\begin{prop}\label{prop:eigen}
For any $\epsilon>0$ and $0\leq{}\alpha\leq{}1/2$, there exists a $\phi_{\alpha}$ such that $\Delta_{S^{n}}\phi_{\alpha}=j(j+n-1)\phi_{\alpha}$ and $\phi_{\alpha}$ is given by,
\begin{equation}\phi_{\alpha}(x)=h^{-\frac{\alpha(n-1)}{2}}\sum_{k=1}^{N_{\alpha}}(x_{1}+iP_{k}(x_{2},\dots{},x_{n}))^{j}\quad{}h^{-2}=j(j+n-1)\label{phialpha}\end{equation}
where  $N_{\alpha}=\tilde{\epsilon} h^{(\alpha-1/2)(n-1)}$  for some small but fixed $\tilde{\epsilon}$ dependent on $\epsilon$ and $P_{k}$ a linear polynomial whose coefficients $\alpha_{k}^{m}$ obey
\begin{enumerate}
\item $|1-\alpha^{2}_{k}|\leq{}{\epsilon}{}h^{2\alpha}$
\item $|\alpha^{m}_{k}|\leq{}{\epsilon}h^{\alpha}\quad{}m\neq{}2.$
\end{enumerate}
Further there are constants $c_{1}$ and $c_{2}$ so that
\begin{equation}c_{1}\leq{}\norm{\phi_{\alpha}}_{L^{2}}\leq{}c_{2}.\label{L2bounds}\end{equation}
\end{prop}

\begin{proof}
We construct $\phi_{\alpha}$ by taking rotations of the standard highest weight harmonic
$$u(x)=(ix_{1}+x_{2})^{j}.$$
For $j=3,\dots,n+1$ we allow the rotation numbers $s_{j}$ to take values in the set $\{h^{1/2}l\mid{}l=1,2,\dots,\lfloor \tilde{\epsilon} h^{\alpha-1/2}\rfloor \}$ (where $\tilde{\epsilon}$ is some small but fixed number). For each $s_{j}$ we define the associated rotation $R_{s_{j}}$ by 
$$(R_{s_{j}}(x))_{2}=\sqrt{1-s_{j}^{2}}x_{2}+s_{j}x_{j}$$
$$(R_{s_{j}}(x))_{j}=-s_{j}x_{2}+\sqrt{1-s_{j}^{2}}x_{j}$$
$$(R_{s_{j}}(x))_{m}=x_{m}\quad{}m\neq{}2,j.$$
Let
$$\phi_{\alpha}=h^{-\frac{\alpha(n-1)}{2}}\sum_{[s_{3},\dots,s_{n+1}]}u\circ{}R_{s_{n+1}}(x)\circ{}R_{s_{n}}\circ\cdots\circ{}R_{s_{3}}.$$
We claim that $\phi_{\alpha}$ has the necessary properties. Each individual term in the summand is an eigenfunction so clearly $\phi_{\alpha}$ is also an eigenfunction. Under the action of each rotation $R_{s_{j}}(x)$, $x_{1}$ is fixed so remains fixed under composition. Writing the $(n-1)$-tuple $S=(s_{3},\dots{}s_{n+1})$ and denoting
$$R_{S}=R_{s_{n+1}}(x)\circ{}R_{s_{n}}\circ\cdots\circ{}R_{s_{3}}$$
we see that
\begin{equation}(R_{S})_{2}=x_{2}\left(\prod_{j=3}^{n+1}\sqrt{1-s_{j}^{2}}\right)+\sum_{k=3}^{n+1}x_{k}s_{k}\left(\prod_{j=k+1}^{n+1}\sqrt{1-s_{j}^{2}}\right).\label{RS2}\end{equation}
Since each $s_{j}$ obeys $|s_{j}|\leq{}\tilde{\epsilon} h^{\alpha}$ by making $\tilde{\epsilon}$ suitably small we obtain the coefficient bounds
$$|1-\alpha^{2}_{k}|\leq{}\epsilon{}h^{2\alpha}$$
$$|\alpha^{m}_{k}|\leq{}\epsilon{}h^{\alpha}\quad{}m\neq{}2.$$
Therefore it remains only to prove the $L^{2}$ estimate. Note that there are $h^{(\alpha-1/2)(n-1)}$ terms in the summand each with $L^{2}$ norm of $h^{\frac{n-1}{4}}$ so \eqref{L2bounds} holds if for $S\neq{}S'$, $u\circ{}R_{S}$ and $u\circ{}R_{S'}$ are suitably orthogonal. We define
$$|S-S'|=\sup_{j}|s_{j}-s_{j}'|$$ 
and claim that for any $N>0$
$$\langle{}u\circ{}R_{S},u\circ{}R_{S'}\rangle\leq{}h^{-\frac{n-1}{2}}\left(1+\frac{|S-S'|}{h^{1/2}}\right)^{-N}.$$
Under  a change of variables
$$x\to{}R_{S'}^{-1}=R_{s'_{3},3}^{-1}\circ \cdots \circ R_{s'_{n+1},n+1}^{-1}$$ 
this reduces to showing that
\begin{equation}\left|\int\left(u\circ{}R_{S}\circ R_{S'}^{-1}(x)\right)u(x)d\mu(x)\right|\leq{}h^{-\frac{n-1}{2}}\left(1-\frac{|S-S'|}{h^{1/2}}\right)^{-N}.\label{orthintest}\end{equation}
From the arguments leading to \eqref{RS2} we can say that
$$u\circ{}R_{S}\circ{}R_{S}^{-1}=(x_{1}+iP_{S,S'}(x_{2},\dots{}x_{n+1}))^{j}$$
where $P_{S,S}$ is a linear polynomial in $x_{2},\dots{}x_{n+1}$. Let $k$ be such that $|s_{k}-s_{k}'|=|S-S'|$ and suppose that we have a lower bound on the $x_{k}$ coefficient, $\alpha^{k}(S,S')$, of
\begin{equation}|\alpha^{k}(S,S')|>c|S-S'|\quad\text{for some }c>0\label{coeffbound}\end{equation}
We will first assume \eqref{coeffbound} and use this to integrate by parts to show that $\eqref{coeffbound}\Rightarrow\eqref{orthintest}$, we then prove \eqref{coeffbound}.  Let $\theta=(\theta_{1},\dots{}\theta_{n})$ be a spherical coordinate system  so that $\theta_{n}\in[0,\pi]$ and the other $\theta_{i}\in[0,2\pi]$ and
$$x_{k}=\cos(\theta_{n})$$
$$x_{k+1}=\sin(\theta_{n})$$
$$\vdots$$
$$x_{2}=\sin(\theta_{n})\cdots{}\sin(\theta_{2})\sin(\theta_{1})$$
$$x_{1}=\sin(\theta_{n})\cdots{}\sin(\theta_{2})\cos(\theta_{1}).$$
Then
$$\frac{\partial (u\circ{}R_{S}\circ{}R_{S'}^{-1})}{\partial\theta_{n}}=j(x_{1}+iP_{S,S'}(x_{2},\cdots,x_{n})^{j-1}\left(F(\theta_{1},\theta_{n-1})\cos(\theta_{n})+i\alpha^{k}(S,S')\sin(\theta_{n})\right).$$
If we are suitable close to the region $\theta_{n}=\pi/2$ and \eqref{coeffbound} holds we have the lower bound
$$\left|F(\theta_{1},\theta_{n-1})\cos(\theta_{n})+i\alpha^{k}(S,S')\sin(\theta_{n})\right|\geq{}\frac{c}{2}|S-S'|$$
and can use this factor to integrate by parts. On the other hand away from the region $\theta_{n}=\pi/2$ we know that $u(\theta)$ decays exponentially so this contribution to the integral must be small. We complete the argument then by cutting the integral over $S^{n}$ into two pieces, one where we may integrate by parts, the other where exponential decay dominates. Let $\chi$ be a smooth cut off function supported in $|\tau|\leq{}2$ and equal to one in $|\tau|\leq{}1$. Consider first
$$\int_{S^{n}}\left(u\circ{}R_{S}\circ{}R_{S'}^{-1}(\theta)\right)u(\theta)\chi\left(\frac{\cos(\theta_{n})}{h^{1/4}|S-S'|^{1/2}}\right)d\mu(\theta)$$
On the support of $\chi$ we can write
$$(\sin(\theta_{n}\cdots{}\sin(\theta_{2})\cos(\theta_{1})+iP_{S,S'}(\theta))^{j}=\frac{1}{j|S-S'|}\frac{\partial }{\partial \theta_{n}}(\sin(\theta_{n}\cdots{}\sin(\theta_{2})\cos(\theta_{1})+iP_{S,S'}(\theta))^{j+1}G(\theta)$$ 
where $|G(\theta)|\leq{}1$. Therefore we can integrate by parts. Any time a derivative hits the cut off function or $u(\theta)$ we loose at worst a factor of $\max(h^{-1/2},h^{-1/4}|S-S'|^{1/2})$ so by repeating the argument $2N$ times we get
$$\left|\int_{S^{n}}\left(u\circ{}R_{S}\circ{}R_{S'}^{-1}(\theta)\right)u(\theta)\chi\left(\frac{|\cos(\theta_{n})}{h^{1/2}|S-S'|^{1/2}}\right)d\mu(\theta)\right|\leq{}\left(1+\frac{|S-S'|}{h^{1/2}}\right)^{-N}\int_{S^{n}}|u(x)|d\mu(x)$$
$$=h^{-\frac{n-1}{2}}\left(1+\frac{|S-S|}{h^{1/2}}\right)^{-N}.$$
Now consider
$$\int_{S^{n}}\left(u\circ{}R_{S}\circ{}R_{S'}^{-1}(\theta)\right)u(\theta)\left(1-\chi\left(\frac{|\cos(\theta_{n})}{h^{1/2}|S-S'|^{1/2}}\right)\right)d\mu(\theta)$$
On the support of $1-\chi$ we have $\cos(\theta_{n})>h^{1/4}|S-S'|^{1/2}$, so $x_{n}^{2}>h^{1/2}|S-S'|$ and
$$|u(x)|=e^{j\log(1-|\bar{x}|^{2})}\leq{}e^{-h^{-1/2}|S-S'|}.$$
So
\begin{multline*}
\left|\int_{S^{n}}\left(u\circ{}R_{S}\circ{}R_{S'}^{-1}(\theta)\right)u(\theta)(1-\chi\left(\frac{|\cos(\theta_{n})}{h^{1/2}|S-S'|^{1/2}}\right))d\mu(\theta)\right|\\
\leq{}e^{-h^{-1/2}|S-S'|}\int_{S^{n}}|u\circ{}R_{S}\circ{}R_{S'}^{-1}|d\mu(x)=h^{-\frac{n-1}{2}}e^{-h^{-1/2}|S-S'|}\end{multline*}
which is a much better estimate than we need.

Now it only remains to ascertain \eqref{coeffbound}. Since $u\circ R_{S}\circ R_{S}^{-1}=(x+ix_{2})^{j}$, $\alpha^{k}(S,S)=0$. Therefore if we expand as a series in $S'$ about $S$,
$$\alpha^{k}(S,S')=\sum_{i=3}^{n+1}\frac{\partial^{2}P_{S,S'}}{\partial s_{i}'\partial{x_{k}}}\Big|_{S=S'}(s_{i}-s_{i}')+O(|S-S'|^{2}).$$
If we write each rotation as a matrix $M_{s_{j}}$ then
$\frac{\partial P_{S,S'}}{\partial x_{k}}$ is given by the first element of
$$V(S,S')=M_{s_{n+1}}\times\cdots{}\times M_{s_{3}}\times M_{s'_{3}}^{-1}\times\cdots\times M_{s'_{n+1}}^{-1}e_{k}$$
where $e_{k}$ is the standard unit vector with $1$ in the entry corresponding to $x_{k}$. Now if $\partial_{s'_{i}}V(S,S')$ is the vector with elements given by  the partial derivative of the elements of $V(S,S')$ with respect to $s'_{i}$,
$$\partial_{s'_{i}}V(S,S')=M_{s_{n+1}}\times\cdots\times M_{s_{3}}\times M_{s'_{3}}^{-1}\times\cdots\times\partial_{s'_{i}}M_{s'_{i}}^{-1}\times\cdots\times M_{s'_{n+1}}^{-1}e_{k}$$
where $\partial_{s'_{i}}M_{s'_{i}}^{-1}$ is the matrix with elements given by the partial derivative of the elements of $M_{s'_{i}}^{-1}$ with respect to $s_{i}'$. So if we evaluate at $S=S'$
$$\partial_{s'_{i}}V(S,S')\Big|_{S=S'}=M_{s_{n+1}}\times\cdots\times M_{s_{i+1}}\times W_{s_{i}}\times M_{s_{i+1}}^{-1}\times \cdots \times M_{s_{n+1}}^{-1}e_{k}$$
where $W_{s_{i}}=M_{s_{i}}\times \partial_{s_{i}}M_{s_{i}}$. First consider the case $i=k$, if $j\neq{}k$ $M_{s_{j}}^{-1}e_{k}=e_{k}$  so
$$\partial_{s_{i}}V(S,S')\Big|_{S=S'}=M_{s_{n+1}}\circ{}M_{s_{k+1}}W_{s_{k}}e_{k}.$$
Since for any $\alpha$, $s_{i}^{2}<\tilde{\epsilon}^{2}$ we can say that
$$\sqrt{1-s_{i}^{2}}=1+O(\tilde{\epsilon})$$
so
$$W_{s_{i}}=\left[\begin{array}{c|ccc|c|ccc}
O(\tilde{\epsilon} )&0&\cdots&0&1+O(\tilde{\epsilon})&0&\cdots&0\\
\hline
0&0&\cdots& 0&0&0&\cdots& 0\\
\vdots&\vdots&&\vdots&\vdots&\vdots&&\vdots\\
0&0&\cdots& 0&0&0&\cdots& 0\\
\hline
-1+O(\tilde{\epsilon})&0&\cdots&0&O(\tilde{\epsilon})&0&\cdots&0\\
\hline
0&0&\cdots& 0&0&0&\cdots& 0\\
\vdots&\vdots&&\vdots&\vdots&\vdots&&\vdots\\
0&0&\cdots& 0&0&0&\cdots& 0\\\\\end{array}\right].$$

So $W_{s_{k}}e_{k}=(1+O(\tilde{\epsilon}),0,\dots,0,O(\tilde{\epsilon}),\dots{},0)$. From \eqref{RS2} we have seen  that multiplication of the matrices $M_{s_{j}}$ produces a matrix with upper left entry
 $\beta$, obeying $|1-\beta|\leq{}\epsilon{}h^{\alpha}$. So the first component of $\partial_{s'_{i}}V(S,S')\Big|_{S=S'}$ has a lower bound of $c>0$. Now consider the case when $i\neq{}k$. 
 We have $M_{s_{k}}e_{k}=(s_{k},0,\cdots{},0,\sqrt{1-s_{k}^{2}},\dots)$. Now the vector $(s_{k},0,\cdots,0)$ has norm bounded by $\tilde{\epsilon}$ and if $i\neq{}k$,  $W_{s_{i}}e_{k}=0$. Since each of the matrices $M_{s_{j}}$ and $W_{s_{i}}$ represent a bounded operator on $\R^{n-1}$ we can say that
 $$\left|\frac{\partial^{2}P_{S,S'}}{\partial s_{i}\partial x_{k}}\Big|_{S'=S}\right|\leq{}\tilde{\epsilon}.$$
 So by choosing $\tlide{\epsilon}$ small enough we have
$$\left|\sum_{i=3}^{n+1}\frac{\partial^{2}P_{S,S'}}{\partial s_{i}'\partial{x_{k}}}\Big|_{S=S'}(s_{i}-s_{i}')\right|\geq{}c|s_{k}-s_{k}'|=c|S-S'|.$$
So
$$|\alpha^{k}(S,S')|>c|S-S'|.$$
Therefore $u\circ{}R_{S}$ and $u\circ{}R_{S'}$ are suitably orthogonal and the $L^{2}$ estimates \eqref{L2bounds} hold.
\end{proof}

Having obtained our combination, $\phi_{\alpha}$ of highest weight harmonics it only remains to prove that there is indeed a $h^{1-2\alpha}\times{}h^{(1-\alpha)(n-1)}$ region where $\phi_{\alpha}$ is large enough.

\begin{prop}\label{linfinity}
Suppose $\phi_{\alpha}$ is given by \eqref{phialpha} . Then there is a $h^{1-2\alpha}\times{}h^{(1-\alpha)(n-1)}$ region in which
$|\phi_{\alpha}|>ch^{-\frac{n-1}{2}+\frac{\alpha(n-1)}{2}}$.
\end{prop}

\begin{proof}
We prove this by expanding $\phi_{\alpha}$ about the point $(\theta_{1},\dots,\theta_{n})=(0,\pi/2,\dots{}\pi/2)$. This corresponds to the point $(1,0,\dots,0)\in\R^{n+1}$ which is fixed by all the rotations so all terms in the sum are equal to $1$ at this point. This point lies on the equator where the original harmonic is equal to $(x_{1}+ix_{2})^{j}=e^{ij\theta_{1}}$ and at $(0,\pi/2,\dots,\pi/2)$, $|e^{-ij\theta_{1}}\phi_{\alpha}|=h^{-\frac{n-1}{2}+\frac{\alpha(n-1)}{2}}.$
When $|\theta_{m}-\pi/2|\leq{}\epsilon{}h^{1/2},m\neq{}1$ the conditions on the coefficients of Proposition \ref{prop:eigen} tell us that for each term in the sum defining $\phi_{\alpha}$,
$$\left|\frac{\partial}{\partial\theta_{1}}e^{-ij\theta_{1}}(\sin(\theta_{n})\cdots{}\sin(\theta_{2})\cos(\theta_{1})+iP_{k}(\theta))^{j}\right|\leq{}jh^{2\alpha}\leq{}h^{2\alpha-1}$$
and for $m\neq{}1$
$$\left|\frac{\partial}{\partial\theta_{m}}e^{-ij\theta_{1}}(\sin(\theta_{n})\cdots{}\sin(\theta_{2})\cos(\theta_{1})+iP_{k}(\theta))^{j}\right|\leq{}jh^{\alpha}\leq{}h^{\alpha-1}.$$
So
$$\left|\frac{\partial}{\partial\theta_{1}}(e^{-ij\theta_{1}}\phi_{\alpha})\right|\leq{}h^{2\alpha-1}\cdot{}h^{-\frac{n-1}{2}+\frac{\alpha(n-1)}{2}}$$
and when $m\neq{}1$
$$\left|\frac{\partial}{\partial\theta_{m}}(e^{-ij\theta}\phi_{\alpha})\right|\leq{}h^{\alpha-1}\cdot{}h^{-\frac{n-1}{2}+\frac{\alpha(n-1}{2}}.$$
So if we take a $h^{1-2\alpha}$ in $\theta_{1}$ by $h^{1-\alpha}$ in the other $\theta_{m}$ region about $(0,\pi/2,\dots,\pi/2)$  we will still have
$$|\phi_{\alpha}|=|e^{-ij\theta_{1}}\phi_{\alpha}|>h^{-\frac{n-1}{2}+\frac{\alpha(n-1)}{2}}.$$
\end{proof}

\section{Predicting the correct scale}\label{pickexample}
In this section we discuss the heuristics of the semiclassical proof. The details of the proof can be found in \cite{koch} and \cite{tacy09} and we will not address them here. The semiclassical approach to eigenfunction estimates is to study quasimodes, that is functions such that
$$\norm{(h^{2}\Delta-1)u}_{L^{2}(M)}\lesssim{}h\norm{u}_{L^{2}(M)}$$
or more generally
$$\norm{p(x,hD)u}_{L^{2}(M)}\lesssim{}h\norm{u}_{L^{2}(M)}$$ where $p(x,hD)$ is a semiclassical pseudodifferential operator,
$$p(x,hD)u=\frac{1}{(2\pi{}h)^{n}}\int{}e^{\frac{i}{h}\langle{}x-y,\xi\rangle}p(x,\xi)u(y)d\xi{}dy$$
 whose symbol, $p(x,\xi)$ satisfies the admissibility criteria
\begin{itemize}
\item[1)] If $p(x_{0},\xi_{0})=0$ then $\nabla_{\xi}p(x_{0},\xi_{0})\neq{}0$
\item[2)] The characteristic set $\{\xi\mid{}p(x_{0},\xi)=0\}$ has positive definite second fundamental form.
\end{itemize}
For Laplacian eigenfunctions the semiclassical symbol $p(x,\xi)=|\xi||_{g}^{2}-1$ so clearly this is admissible. In fact in the flat case the characteristic set is the $n-1$ sphere (the canonical example of a hypersurface with positive definite second fundamental form). Quasimodes, as distinct from eigenfunctions, have the nice property that they remain quasimodes under localisation so we may work locally. It is relatively easy to show that contributions localised away from the characteristic set are small. Therefore we may work locally around some point $(x_{0},\xi_{0})$ such that $p(x_{0},\xi_{0})=0$. To prove $L^{p}$ estimates we perform the following steps

\begin{itemize}
\item[Step 1)] Factorise the symbol. Since the characteristic set is non-degenerate (by admissibility condition 1) we can always find some $\xi_{i}$ such that
$$|\partial_{\xi_{1}}p(x_{0},\xi_{0})|>c>0$$
so by the implicit function theorem, locally
$$p(x,\xi)=e(x,\xi)(\xi_{i}-a(x,\xi'))$$
where $|e(x,\xi)|>c>0$. The semiclassical calculus then tells us we may invert $e(x,hD)$ to obtain
$$(hD_{x_{i}}-a(x,hD_{x'}))u=hf$$
where $\norm{f}_{L^{2}(M)}\lesssim{}\norm{u}_{L^{2}(M)}$.
\item[Step 2)]
By setting $x_{i}=t$ we find that $u$ is an approximate solution to the semiclassical evolution equation
$$(hD_{t}-a(t,x',hD_{x'}))v(t,x)=0.$$
Therefore by Duhammel's principle we may write
$$u=U(t,0)u(0,x')+\int_{0}^{t}U(t-\tau,\tau)f(\tau)d\tau$$
where $U(t,\tau)$ satisfies
$$\begin{cases}
(hD_{t}-a_{1}(\tau+t,x',hD_{x'}))U_{h}(t,\tau)=0,\\
U_{h}(0,\tau)=\Id.\end{cases}$$
The problem then reduces to finding (uniform in $\tau$) $L^{2}\to{}L^{p}$ mapping norms of $U(t,\tau)$ or the restriction of $U(t,\tau)$ to a submanifold.
\item[Step 3)]
We estimate the $L^{2}(M)\to{}L^{p}(X)$ norms through a $TT^{\star}$ method. The key point is to obtain estimates of the form
$$\norm{U(t,\tau)U(s,\tau)^{\star}}_{L^{1}(X)\to{}L^{\infty}(X)}\lesssim{}h^{-\kappa_{\infty}}(h+|t-s|)^{-\gamma_{\infty}}$$
and
$$\norm{U(t,\tau)U(s,\tau)^{\star}}_{L^{2}(X)\to{}L^{2}(X)}\lesssim{}h^{-\kappa_{2}}(h+|t-s|)^{-\gamma_{2}}.$$
All other estimate follow by interpolation with these and resolving the $t-s$ integral with either Young's inequality or Hardy-Littlewood-Sobolev. It is here we see the connection with Keel-Tao \cite{keel} abstract Strichartz estimates which can be proved in the same fashion.
\end{itemize}

For submanifold estimates there is an additional question of whether this special direction, $x_{i}=t$, lies along the submanifold or not. It turns out we may assume that it does as this case gives all sharp estimates. That is if $\Sigma=\{(y,z)\in{}M\mid{}z=0\}$ we may assume that $\xi_{i}$ is dual to $y_{1}$. 

The interpolation argument of Step 3 gives an estimate of the form
$$\norm{U(t)U(s)^{\star}}_{L^{p'}\to{}L^{p}}\lesssim{}h^{-\kappa_{p}}(h+|t-s|)^{-\gamma_{p}}.$$
We can think of this as a decay estimate for propagation time $|t-s|$. To generate sharp examples we then need to find what scale of $|t-s|$ makes the largest contribution to the estimate. The sharp example will then be the $T_{\alpha}$ whose long direction is equal to this critical scale $|t-s|_{c}$. That is $|t-s|_{c}=h^{1-2\alpha}.$

Therefore the regime changes in the $L^{p}$ estimates depend only on the power $\gamma_{p}$ the numerology of which depends only the $L^{1}(X)\to{}L^{\infty}(X)$ estimates and the $L^{2}(X)\to{}L^{2}(X)$ estimates. To resolve the $t-s$ integral we estimate
$$\int{}(h+|\tau|)^{-\frac{\gamma_{p}p}{2}}d\tau.$$
\begin{itemize}
\item If $\frac{\gamma_{p}p}{2}>1$ the major contribution comes from the smallest possible $\tau=\tau_{min}$ 
\item If $\frac{\gamma_{p}p}{2}<1$ the major contribution comes from the largest possible $\tau_{max}$ 
\end{itemize}
In both cases we expect the sharp examples to be given by $T_{\alpha_{min}}$ and $T_{\alpha_{max}}$ where
$$\tau_{min}=h^{1-\alpha_{min}}\quad{}\tau_{max}=h^{1-\alpha_{max}}.$$
Independent of $X$ we can obtain a $L^{1}(X)\to{}L^{\infty}(X)$ estimate of
\begin{equation}\norm{U(t,\tau)U(s,\tau)^{\star}}_{L^{1}(X)\to{}L^{\infty}(X)}\lesssim{}h^{-\frac{n-1}{2}}(h+|t-s|)^{-\frac{n-1}{2}}\label{L1Linf}\end{equation}
so the key point is to obtain the $L^{2}(X)\to{}L^{2}(X)$ estimates. In \cite{tacy09} we see that these are given by the $L^{2}(X)\to{}L^{2}(X)$ mapping norms of an operator
$$W(t-s)u=\int{}W(x,y,t-s)u(y)dy,$$
$$W(x,y,t-s)=h^{-\frac{n-1}{2}}(h+|t-s|)^{-\frac{n-1}{2}}e^{\frac{i}{h}\phi(x,y,t-s)}b(t,s,x,y)$$
where the factor 
$$e^{\frac{i}{h}\phi(x,y,t-s)}$$
oscillates with frequency $h^{-1}|t-s|^{-1}$. From considerations of almost orthogonality we expect that the $L^{2}(X)\to{}L^{2}(X)$ mapping norm of such an operator should be determined by the mapping norm on $h^{1/2}|t-s|^{1/2}$ boxes. This suggests a general heuristic for finding those $p$ at which the behaviour of the $L^{2}(M)\to{}L^{p}(X)$ estimates change.
\begin{enumerate}
\item Calculate the $L^{2}(X)\to{}L^{2}(X)$ mapping norm of $U(t,\tau)U(s,\tau)^{\star}$ on the intersection of a $h^{1/2}|t-s|^{1/2}$ box with $X$. 
\item Interpolate that result with the $L^{1}(X)\to{}L^{\infty}(X)$ estimate given by \eqref{L1Linf}. This will give $\gamma_{p}$ for all p.
\item Find the values of $p$ for which $\frac{\gamma_{p}p}{2}=1$. We expect regime changes at these $p$. 
\item Determine $\tau_{min}$ and $\tau_{max}$ for each critical $p$. The functions $T_{\alpha_{min}}$ and $T_{\alpha_{max}}$ are expected to give sharp examples.
\end{enumerate}

\subsection{Whole and submanifold estimates}
We apply the heuristic an consider the $L^{2}(X)\to{}L^{2}(X)$ norm on the intersection between $X$ and a $h^{1/2}|t-s|^{1/2}$ box. We obtain, for $X$ a $k$ dimensional submanifold.
$$\norm{W(t-s)}_{L^{2}(X)\to{}L^{2}(X)}\lesssim{}h^{-\frac{n-1}{2}}(h+|t-s|)^{-\frac{n-1}{2}}(h^{\frac{1}{2}}|t-s|^{1/2})^{k-1}=h^{-\frac{n-k}{2}}(h+|t-s|)^{-\frac{n-k}{2}}$$
From the interpolation numerology we obtain that for whole manifolds and hypersurfaces there is only one $p$ so that $\frac{\gamma_{p}p}{2}=1$ ($p=\frac{2(n+1)}{n-1}$ and $p=\frac{2n}{n-1}$ respectively). For a lower dimensional submanifold $\frac{\gamma_{p}p}{2}\geq{}1$ for all $p\geq{}2$. Since we truncate at $|t-s|\leq{}h$ the smallest effective scale is $\tau_{min}=h$ and since we are dealing with compact sets $\tau_{max}=1$. Therefore our sharp examples will come from $T_{0}$ and $T_{1/2}$ for the whole manifold and hypersurface case and from $T_{0}$ alone for the lower submanifolds. 

\subsection{$\Sigma_{\beta}$ estimates}
By considering $\Sigma_{\beta}$ we introduce a new scale (namely $h^{\beta}$) into the problem. If $h^{\beta}\geq{}h^{1/2}|t-s|^{1/2}$ a $h^{1/2}|t-s|^{1/2}$ box can lie fully in $\Sigma_{\beta}$ and therefore we get the $L^{2}\to{}L^{2}$ estimate
$$\norm{W(t-s)}_{L^{2}(\Sigma_{\beta})\to{}L^{2}(\Sigma_{\beta})}\lesssim{}1,\quad{}\quad|t-s|\leq{}h^{2\beta-1}.$$
Which is the same as over the whole manifold. If on the other hand $h^{\beta}\leq{}h^{1/2}|t-s|^{1/2}$, the $h^{1/2}|t-s|^{1/2}$ box does not lie fully in $\Sigma_{\beta}$ so we obtain
$$\norm{W(t-s)}_{L^{2}(\Sigma_{\beta})\to{}L^{2}(\Sigma_{\beta})}\lesssim{}h^{-\frac{n-k}{2}+\frac{\beta(n-k)}{2}}(h+|t-s|)^{-\frac{n-k}{2}},\quad{}|t-s|\geq{}h^{2\beta-1}.$$
Therefore we potentially have two points at which $\frac{\gamma_{p}p}{2}=1$.  Where $\Sigma$ is a hypersurface there are two critical points. The first arises from the $|t-s|\leq{}h^{2\beta-1}$ estimates and is at $p=\frac{2(n+1)}{n-1}$. Therefore for this critical point $\tau_{min}=h$ and $\tau_{max}=h^{2\beta-1}$. The second point arises from the $|t-s|\geq{}h^{2\beta-1}$ estimate and is at $p=\frac{2n}{n-1}$. For this critical point $\tau_{min}=h^{2\beta-1}$ and $\tau_{max}=1$. Therefore sharp behaviour should be determined by $T_{0}$, $T_{1-\beta}$ and $T_{1/2}$. For lower dimensional submanifolds we obtain a critical $p$ coming from the $|t-s|\leq{}h^{2\beta-1}$ estimate again with $\tau_{min}=h$, $\tau_{max}=h^{2\beta-1}$. However from the long time estimates we always have $\frac{\gamma_{p}p}{2}\geq{}1$, therefore we only need to examine $\tau_{min}$, in this case $h^{2\beta-1}$. Therefore sharp examples should come from $T_{0}$ and $T_{1-\beta}$ alone.

\bibliography{references}
\bibliographystyle{abbrv} 

\end{document}